\documentclass[a4paper,10pt,reqno]{amsart}

\usepackage[utf8]{inputenc}
\usepackage{ulem}
\usepackage[T1]{fontenc}
\usepackage{amsthm}
\usepackage{amsmath}
\usepackage{amssymb}
\usepackage[inline]{enumitem}
\usepackage{comment}
\PassOptionsToPackage{hyphens}{url}\usepackage{hyperref}
\usepackage{fancyhdr}
\usepackage{mathrsfs}
\usepackage{stmaryrd}
\usepackage{tikz-cd}
\usetikzlibrary{arrows}
\usepackage[usestackEOL]{stackengine}
\usepackage{scalerel}
\usepackage{array}

\theoremstyle{definition}

\newtheorem{theorem}{Theorem}[section]
\newtheorem{lemma}[theorem]{Lemma}
\newtheorem{proposition}[theorem]{Proposition}
\newtheorem{corollary}[theorem]{Corollary}

\theoremstyle{definition}
\newtheorem{definition}[theorem]{Definition}

\newtheorem{remark}[theorem]{Remark}

\newtheorem{examples}[theorem]{Examples}
\newtheorem{remarks}[theorem]{Remarks}

\definecolor{blue-url}{RGB}{0,0,100}
\definecolor{red-url}{RGB}{100,0,0}
\definecolor{green-url}{RGB}{0,100,0}
\definecolor{light-yellow}{RGB}{255,255,128}
\definecolor{light-blue}{RGB}{193,255,255}
\definecolor{light-red}{RGB}{239,83,80}

\hypersetup{
	pdftitle={On the arithmetic of power monoids},
	pdfauthor={},
	pdfmenubar=false,
	pdffitwindow=true,
	pdfstartview=FitH,
	colorlinks=true,
	linkcolor=blue-url,
	citecolor=green-url,
	urlcolor=red-url
}

\renewcommand{\emptyset}{\varnothing}
\renewcommand{\setminus}{\smallsetminus}
\renewcommand{\,}{\kern 0.1em}

\providecommand\llb{\llbracket}
\providecommand\rrb{\rrbracket}
\providecommand\sqeq{\sqsubseteq}

\newcommand{\evid}[1]{\textsf{#1}}
\newcommand{\fin}{\mathrm{fin}}

\makeatletter
\DeclareFontFamily{OMX}{MnSymbolE}{}
\DeclareSymbolFont{MnLargeSymbols}{OMX}{MnSymbolE}{m}{n}
\SetSymbolFont{MnLargeSymbols}{bold}{OMX}{MnSymbolE}{b}{n}
\DeclareFontShape{OMX}{MnSymbolE}{m}{n}{
	<-6>  MnSymbolE5
	<6-7>  MnSymbolE6
	<7-8>  MnSymbolE7
	<8-9>  MnSymbolE8
	<9-10> MnSymbolE9
	<10-12> MnSymbolE10
	<12->   MnSymbolE12
}{}
\DeclareFontShape{OMX}{MnSymbolE}{b}{n}{
	<-6>  MnSymbolE-Bold5
	<6-7>  MnSymbolE-Bold6
	<7-8>  MnSymbolE-Bold7
	<8-9>  MnSymbolE-Bold8
	<9-10> MnSymbolE-Bold9
	<10-12> MnSymbolE-Bold10
	<12->   MnSymbolE-Bold12
}{}

\let\llangle\@undefined
\let\rrangle\@undefined
\DeclareMathDelimiter{\llangle}{\mathopen}%
{MnLargeSymbols}{'164}{MnLargeSymbols}{'164}
\DeclareMathDelimiter{\rrangle}{\mathclose}%
{MnLargeSymbols}{'171}{MnLargeSymbols}{'171}
\makeatother

\begin{document}
\title{On the arithmetic of power monoids}
\author{Laura Cossu}
\address{(L.~Cossu) Department of Mathematics and Computer Science, University of Cagliari | Palazzo delle Scienze, Via Ospedale 72, 09124 Cagliari, Italy}
\email{laura.cossu3@unica.it}
\urladdr{https://sites.google.com/view/laura-cossu}
\author{Salvatore Tringali}
\address{(S.~Tringali) School of Mathematical Sciences,
Hebei Normal University | Shijiazhuang, Hebei province, 050024 China}
\email{salvo.tringali@gmail.com}
\urladdr{http://imsc.uni-graz.at/tringali}

%
%
%
\begin{abstract}
Given a monoid $H$ (written multiplicatively), the family 
$\mathcal{P}_{\mathrm{fin},1}(H)$ of all non-empty finite subsets of $H$ 
containing the identity element $1_H$ is itself a monoid, called the 
reduced finitary power monoid of $H$, under the operation of setwise 
multiplication induced by $H$.

We investigate the arithmetic of $\mathcal P_{\fin,1}(H)$  from the perspective of minimal factorizations into irreducibles, paying particular attention to the potential presence of non-trivial idempotents. 
Among other results, we provide necessary and sufficient conditions on $H$ for $\mathcal P_{\fin,1}(H)$ to admit unique minimal factorizations. 
Our results generalize and shed new light on recent developments on the topic.
\end{abstract}
\maketitle
\thispagestyle{empty}

\section{Introduction}
\label{sec:intro}
Roughly speaking, factorization theory is the study of a broad spectrum of phenomena arising from the possibility or impossibility of extending the fundamental theorem of arithmetic from the integers to more abstract settings where a set comes endowed with a (binary) operation.

The theory has traditionally focused on domains and cancellative monoids, primarily in the commutative setting \cite{GeHK06, Got-And-2022}. It is only in the past few years that researchers have made first significant steps towards a systematic study of factorizations in non-cancellative (and non-commutative) settings; see \cite[Section~2.4]{An-Tr18} for a review of some older ``scattered results'' along the same lines. The new trend has partly originated from a renewed interest in \evid{power monoids}~\cite{Fa-Tr18, An-Tr18}, a class of ``highly non-cancellative'' monoids intensively studied by semigroup theorists and computer scientists during the 1980s and 1990s, and featuring a dense interplay between algebra and combinatorics (see \cite[Section 1]{GaTr25}, \cite[Section 1]{Tri-2023(c)}, and references therein for a more extensive historical overview). 

Power monoids exhibit a rich arithmetic, making them an ideal testing ground for the on\-going development of a ``unifying theory of factorization'' \cite{Tr20(c), Co-Tr-21(a), Cos-Tri-2023(a), Cos-Tri-2023(b), Co25} that extends far beyond the boundaries of the classical theory. Moreover, they provide an algebraic framework for a variety of interesting problems in ad\-di\-tive combinatorics and closely related areas \cite{Tri-Yan2023(a), Bie-Ger-22, Casab-Danna-GarSan-2023, Tri-Yan2023(b), GaTr25}, notably including S\'{a}rk\"ozy's conjecture \cite[Conjecture 1.6]{Sark2012} on the ``additive irreducibility'' of the set of [non-zero] squares in a finite field of prime order. 
Further recent contributions to the topic are due to Gotti and his students \cite{Gonz-Li-Rabi-Rodr-Tira-2025, Agg-Got-Lu-2025, Dan-Got-Hong-Li-Sch-2025}, who have investigated the arithmetic of power monoids from the perspective of algebraic combinatorics; and to Reinhart \cite{Rein-2025}, who has taken the first significant step in years towards the proof of a challenging conjecture of Fan and Tringali \cite[Sect.~5]{Fa-Tr18}.

In more detail, let $H$ be a (multiplicatively written) monoid, often referred to in this context as the \textit{ground monoid}. Equipped with the operation of setwise multiplication
$$
(X, Y) \mapsto XY := \{xy \colon x \in X, \, y \in Y\},
$$
the family of all non-empty finite subsets of $H$ containing the identity $1_H$ forms a monoid in its own turn, herein denoted by $\mathcal P_{\fin,1}(H)$ and called the \evid{reduced} (\evid{finitary}) \evid{power monoid} of $H$. 
It is a natural problem to understand if, and how, a set $X \in \mathcal P_{\fin,1}(H)$ can be decomposed into a product of certain other sets regarded as the ``building blocks'' of the decomposition process.

In the classical theory, the building blocks are typically the atoms of the monoid under consideration, an atom being a non-unit that does not factor as a product of two non-units. However, atoms become less appealing in non-cancellative contexts, where the existence itself of atomic decompositions is not guaranteed even in monoids with ``smooth properties''.
Inspired by Anderson and Valdes-Leon's work \cite{And-ValLeo-1996} in the commutative setting, this has led Tringali \cite{Tr20(c)} to the idea of refining the notion of atom with the weaker notion of \textit{irreducible} (Section~\ref{subsec:irreds-atoms-quarks}). It is worth stressing from the outset that, while the terms ``atom'' and ``irreducible'' are used
interchangeably in the classical theory, they assume distinct meanings in our framework: the two concepts coincide, say, in a cancellative commutative monoid (see Remark \ref{rem:dedekind-finiteness-and-acyclicity}\ref{rem:dedekind-finiteness-and-acyclicity(1)}), but the latter is significantly more versatile than the former and better suited for a more general version of the theory (see, for instance, \cite[Theorem 3.10]{Tr20(c)}, \cite[Theorem 5.19]{Co-Tr-21(a)}, and \cite[Theorems 4.7 and 5.1]{Cos-Tri-2023(a)}). 

In fact, there is more to the story, as another problem is inherent to the
non-cancellative setting and also arises in the non-commutative (cancellative)
scenario: in many situations (e.g., in the presence of idempotents other than the identity), factorizations blow up in a predictable fashion, causing most of the invariants studied in the classical theory to lose their significance. 
This prompted Antoniou and Tringali \cite[Section~4]{An-Tr18}, followed by Cossu and Tringali \cite{Cos-Tri-2023(a)}, to introduce and develop the notion of \textit{minimal fac\-tor\-i\-za\-tion} (Section~\ref{subsec: (minimal) factorizations}) as a natural refinement of classical factorizations capable of countering the ``blow-up phenomenon''.

In the present paper, we add to this line of research. The goal is to study minimal factorizations into irreducibles in $\mathcal P_{\fin,1}(H)$.
In Section \ref{sec:irreds-factor}, after recalling basic ideas and definitions, we introduce the notion of (minimal) factorization in an arbitrary monoid $M$, describing its building blocks  and related properties. Next, we specialize these general concepts to the case of reduced power monoids. In particular, Proposition \ref{prop:2.3} characterizes the irreducibles of $\mathcal P_{\fin,1}(H)$ that are not atoms. In Propositions \ref{prop:2.7} and  \ref{prop:2.8}, we gather a series of finiteness results on (minimal) factorizations that hold for arbitrary reduced power monoids. In Section \ref{sect:3}, we investigate conditions on $H$ under which $\mathcal P_{\fin,1}(H)$ is UmF, i.e., it admits ``unique'' minimal factorizations. The characterization provided by the main Theorem \ref{thm:UmFness} implies that the UmF-ness of reduced power monoids can be reduced to the UmF-ness of reduced power monoids over a particular class of monoids we refer to as ``almost-breakable'' (Definition \ref{dfn: almost-brakable}). Section \ref{sec: UmFness and almost-breakable} deals with reduced power monoids of almost-breakable monoids. Most notably, we obtain a complete characterization of UmF-ness when $H$ contains at least one non-idempotent non-unit (Corollary \ref{cor: H cancellative}) or is commutative (Theorem \ref{thm: H commutative}). 
Our results generalize those in \cite{An-Tr18} (which only deal with factorizations into atoms) and make it possible to avoid imposing unnatural conditions to guarantee the existence of factorizations.

\section{Irreducibles and factorization}
\label{sec:irreds-factor}

Through this section, we introduce the notions of factorization we are interested in and specialize them to the class of monoids we aim to study.
We address the reader to Clifford and Preston's classical monograph \cite{Cli-Pre-1961} for generalities on semigroups and monoids. 

Notation and terminology, if not explained upon first use, are standard or should be clear from the context. In particular, we denote by $\mathbb N$ the non-negative integers and, given $a, b \in \mathbb N$, we refer to $\llb a, b \rrb := \{x \in \mathbb N \colon a \le x \le b\}$ as the \evid{discrete interval} from $a$ to $b$. Also, as a rule of thumb, we use the letter $M$ for an arbitrary monoid, and $H$ for the ground monoid of a reduced power monoid.

\subsection{Irreducibles, atoms, and quarks}\label{subsec:irreds-atoms-quarks}

Let $M$ be a (multiplicatively written) monoid; we use $1_M$ for the \evid{identity} of $M$. We call an element $x \in M$ a \evid{unit} if $xy = yx = 1_M$ for some $y \in M$; otherwise, $x$ is a \evid{non-unit}. We denote by $M^\times$ the set of units of $M$. It is a basic fact that $M^\times$ is a subgroup of $M$. We call $M$ \evid{reduced} if $M^\times = \{1_M\}$, and \evid{Dedekind-finite} if $xy=1_M$ for some $x, y \in M$ implies $yx = 1_M$.

We denote by $\mid_M$ the binary relation on $M$ defined by $x \mid_M y$ (read as “$x$ divides $y$ in $M$” or “$x$ is a divisor of $y$ in $M$”) if and only if $y \in MxM := {uxv : u, v \in M}$. It is straightforward that $\mid_M$ is reflexive and transitive, hence a (partial) preorder. Accordingly, we refer to $\mid_M$ as the \evid{divisibility preorder} on $M$. Two elements $x, y \in M$ are \evid{associated} (in $M$), written $x \simeq_M y$, if they divide each other. Moreover, $x$ is a \evid{proper divisor} of $y$ (or, equivalently, $x$ properly divides $y$) if $x$ divides but is not associated to $y$ (i.e., $y \in MxM$ and $x \notin MyM$), and is a \evid{unit-divisor} (resp., a \evid{non-unit-divisor}) if $x$ divides (resp., does not divide) the identity $1_M$. Later, we will simply write $\mid$ instead of $\mid_M$ and $\simeq$ instead of $\simeq_M$ whenever the ``ambient monoid'' $M$ is clear from the context.

Following \cite[Section~3]{Co-Tr-21(a)}, we let an \evid{irreducible} (of $M$) be a non-unit-divisor $a \in M$ such that $a \ne xy$ for all non-unit-divisors $x, y \in M$ that properly divide $a$. An \evid{atom} is, on the other hand, a non-unit-divisor that does not factor as a product of two non-unit-divisors, while a \evid{quark} is a non-unit-divisor which is not properly divided by any non-unit-divisor. 

\begin{remarks}
\label{rem:dedekind-finiteness-and-acyclicity}
\begin{enumerate*}[label=\textup{(\arabic{*})},mode=unboxed]
\item\label{rem:dedekind-finiteness-and-acyclicity(1)} 
In a monoid $M$, every atom is an irreducible, and so is every quark. The converse, however, is false in general (see \cite[Example 4.8]{Tr20(c)} and Proposition \ref{prop:2.3} below); but it holds, for instance, when $M$ is \evid{acyclic} \cite[Definition 4.2 and Corollary 4.4]{Tr20(c)}, meaning that $x \ne \allowbreak uxv$ for all $u, v, x \in M$ such that $u$ or $v$ is a non-unit (e.g., a cancellative commutative monoid is acyclic). Note, in addition, that if $M$ is Dedekind-finite, then the unit-divisors of $M$ coincide with its units. In this case, an atom of $M$ is precisely a non-unit that cannot be expressed as a product of two non-units, in line with the definition given in the introduction. Moreover, every acyclic monoid is Dedekind-finite \cite[Corollary 4.4]{Tr20(c)}.
\end{enumerate*}

\vskip 0.05cm

\begin{enumerate*}[label=\textup{(\arabic{*})},mode=unboxed, resume]
\item\label{rem:dedekind-finiteness-and-acyclicity(2)} The notion of \textit{irreducible} used in the present work is actually a special case of a much more abstract concept (see \cite[Definition 3.6]{Tr20(c)} and \cite[Definition 3.1]{Co-Tr-21(a)}), whose definition is based on the idea of equipping a monoid $M$ with a preorder $\preceq$ (in our case, $\preceq$ is the divisibility preorder $\mid_M$). From this perspective, the identity $1_M$ plays a particularly distinguished role, as one first defines a \evid{$\preceq$-unit} as 
an element $u \in M$ such that $1_M \preceq u \preceq 1_M$, and a $\preceq$-non-unit as an element $u \in M$ that is not a $\preceq$-unit. Accordingly, a \evid{$\preceq$-irreducible} 
is a $\preceq$-non-unit $a \in M$ such that $a \ne xy$ for all $\preceq$-non-units 
$x, y \in M$ with $x \prec a$ and $y \prec a$, where $u \prec v$ means $u \preceq v \not\preceq u$. Likewise, a \evid{$\preceq$-atom} is a $\preceq$-non-unit $a \in M$ 
such that $a \ne xy$ for all $\preceq$-non-units $x, y \in M$; and a \evid{$\preceq$-quark} is a $\preceq$-non-unit $a \in M$ such that if $b \prec a$ then $b$ is a $\preceq$-unit. 
\end{enumerate*}
\end{remarks}

We aim to characterize the irreducibles, atoms, and quarks of reduced power monoids. To this end, our first lemma collects some basic properties that will come in handy later.

\begin{lemma}\label{lem: atoms-irreds}
Let $H$ be a monoid and let $X,Y,Z\in \mathcal{P}_{\mathrm{fin},1}(H)$. The following hold:

\begin{enumerate}[label=\textup{(\roman{*})}]

\item\label{lem: atoms-irreds(i)} If $X$ is a divisor of $Y$ in $\mathcal{P}_{\mathrm{fin},1}(H)$, then $X \subseteq Y$.

\item\label{lem: atoms-irreds(ii)} $\mathcal P_{\fin,1}(H)$ is a reduced, Dedekind-finite monoid.

\item\label{lem: atoms-irreds(iii)} $X$ and $Y$ are associated in $\mathcal{P}_{\mathrm{fin},1}(H)$ if and only if $X = Y$.

\item\label{lem: atoms-irreds(iv)} $X$ is irreducible in $ \mathcal{P}_{\mathrm{fin},1}(H)$ if and only if $X \ne \{1_H\}$ and $X\ne YZ$ for all $Y, Z \subsetneq X$.

\item\label{lem: atoms-irreds(v)} $X$ is irreducible in $ \mathcal{P}_{\mathrm{fin},1}(H)$ if and only if it is a quark.

\item\label{lem: atoms-irreds(vi)} If $X$ is irreducible in $\mathcal{P}_{\mathrm{fin},1}(H)$ but not an atom, then $X^2=X$.
\end{enumerate}
\end{lemma}

\begin{proof}
As for \ref{lem: atoms-irreds(i)}, it is enough to note that, if $X$ is a divisor of $Y$ in $\mathcal P_{\fin,1}(H)$, then $Y = UXV$ for some $U, V \in \mathcal{P}_{\mathrm{fin},1}(H)$ and hence $X = \{1_H\} X \{1_H\} \subseteq UXV = Y$. \ref{lem: atoms-irreds(ii)} is now immediate, because $XY = \{1_H\}$ implies by \ref{lem: atoms-irreds(i)} that $X \cup Y \subseteq \{1_H\}$ and hence $X = Y = \{1_H\}$.
On the other hand, \ref{lem: atoms-irreds(iv)} is straightforward from \ref{lem: atoms-irreds(i)}--\ref{lem: atoms-irreds(iii)}, and \ref{lem: atoms-irreds(v)} follows from \cite[Proposition 4.11(iii)]{Tr20(c)}. So, we focus on \ref{lem: atoms-irreds(iii)} and \ref{lem: atoms-irreds(vi)}.

\vskip 0.05cm

\ref{lem: atoms-irreds(iii)} We prove the equivalent statement: $X$ is a proper divisor of $Y$ if and only if $X\mid Y$ and $X \subsetneq Y$. Assume that $X$ is a proper divisor of $Y$ in $\mathcal{P}_{\mathrm{fin},1}(H)$. If $X=Y$, then $Y=\{1_H\}X$, and consequently $Y\mid X$, a contradiction. Thus, $X\ne Y$. Now assume that $X\mid Y$ and $X\ne Y$. By item \ref{lem: atoms-irreds(i)}, $X\subseteq Y$ and $Y \nmid X$, otherwise $Y\subseteq X$, yielding $X=Y$.

\vskip 0.05cm

\ref{lem: atoms-irreds(vi)} Assume that $X$ is irreducible, but it is not an atom. Then, by item \ref{lem: atoms-irreds(ii)}, $X=YZ$ for some non-identity elements $Y,Z\in \mathcal{P}_{\mathrm{fin},1}(H)$. This implies that $Y\mid X$ and $Z\mid X$, but we know from \ref{lem: atoms-irreds(v)} that $X$ is a quark, and this forces $Y=Z=X$.
\end{proof}

In supplement to items \ref{lem: atoms-irreds(v)} and \ref{lem: atoms-irreds(vi)} of Lemma \ref{lem: atoms-irreds}, we continue with a characterization of the irreducibles of a reduced power monoid that are not atoms.

\begin{proposition}
\label{prop:2.3}
Let $H$ be a monoid. A set $X \in \mathcal{P}_{\mathrm{fin},1}(H)$ is irreducible but not an atom if and only if $X = \{1_H,x\}$ for some $x \in H$ such that $x^2 = 1_H$ or $x^2 = x$.
\end{proposition}

\begin{proof}
To begin, assume $X$ is irreducible but not an atom. By items \ref{lem: atoms-irreds(ii)} and \ref{lem: atoms-irreds(vi)} of Lemma~\ref{lem: atoms-irreds}, this yields $X^2 = X \ne \{1_H\}$. Accordingly, $X=\{1_H,x_1,\dots,x_n\}$ for some $x_1, \allowbreak \ldots, \allowbreak x_n \in \allowbreak H \setminus \{1_H\}$. Then $X \subseteq \allowbreak X \{1_H,x_1\} \subseteq X^2 = X$ and hence $\{1_H, x_1\} \mid X \{1_H, x_1\} = X$. Since $X$ is a quark by Lemma~\ref{lem: atoms-irreds}\ref{lem: atoms-irreds(v)}, it follows that $X = \{1_H, x_1\}$ and $\{1_H, x_1\}^2 = \{1_H, x_1\}$, which is only possible if $x_1^2 = x_1$ or $x_1 = 1_H$.

Conversely, suppose that $X=\{1_H,x\}$ for some $x \in H$ such that $x^2=x$ or $x^2=1_H$. It is then clear that $X$ is an idempotent. Thus we are done, because every $2$-element set in $\mathcal{P}_{\mathrm{fin},1}(H)$ is irreducible (e.g., by item \ref{lem: atoms-irreds(iv)} of Lemma \ref{lem: atoms-irreds}) and it is obvious that, in any monoid, an atom is not an idempotent.
\end{proof}

As one might expect, a comprehensive  characterization of the irreducibles of $ \mathcal{P}_{\mathrm{fin},1}(H)$ is going to critically depend on specific properties of the monoids $H$. Nevertheless, there are ``constructive results'' that hold in full generality, as with the following.

\begin{proposition}\label{prop:antichains-and-irreds}
Let $H$ be a monoid and $A \subseteq H \setminus \{1_H\}$ be a non-empty finite \evid{$\mid_H$-antichain}, meaning that $a \nmid_H b$ for all $a, b \in A$ with $a \ne b$. The set $\{1_H\} \cup \allowbreak A$ is then irreducible in $\mathcal P_{\mathrm{fin},1}(H)$.
\end{proposition}

\begin{proof}
The set $A' := \{1_H\} \cup A$ is a non-unit of $\mathcal P_{\mathrm{fin},1}(H)$, because $A$ is a non-empty finite subset of $H\setminus \{1_H\}$ and $\mathcal P_{\mathrm{fin},1}(H)$ is reduced by Lemma \ref{lem: atoms-irreds}\ref{lem: atoms-irreds(ii)}. Suppose for a contradiction that $A'$ is not irreducible. There then exist \textit{proper} subsets $X$ and $Y$ of $A'$, both containing the identity $1_H$, such that $A' = \allowbreak XY$, and this can only happen if $|A| \ge 2$. Consequently, we can find $a, b \in A$ with $a \notin X$ and $b \notin Y$; moreover, each element in $A$ is a non-unit-divisor, or else $A$ would not be a $\mid_H$-antichain (which is absurd) since a unit-divisor divides any other element in $H$ (and $A$ has at least two elements). It follows that $b \in X$ and hence $a \ne b$ (if $b \notin X$, then $b \in XY$ and $b \notin Y$ would yield $b = cd$ for certain $c, d \in A \setminus \{1_H, b\}$, contradicting that $A$ is a $\mid_H$-antichain). In a similar way, $a \in Y$. Putting the pieces together, we can thus conclude that $1_H \ne ba \in A$ and hence $ba = c$ for some $c \in A$, which is impossible as it implies that either $a \mid_H c$ and $a \ne c$, or $b \mid_H c$ and $b \ne c$ (again contradicting that $A$ is a $\mid_H$-antichain).
\end{proof}

\subsection{Factorizations and minimal factorizations}
\label{subsec: (minimal) factorizations}

We write $\mathscr F(X)$ for the free monoid over a set $X$. We refer to the elements of $\mathscr F(X)$ as \evid{$X$-words} and to its identity  as the \evid{empty word}. 

Given a monoid $M$ and an element $x \in M$, we denote by $\mathscr I(M)$ the set of irreducibles of $M$ and define a \evid{factorization} (\evid{into irreducibles}) of $x$ as an $\mathscr I(M)$-word $\mathfrak a$ such that $\pi_M(\mathfrak a) = x$. Here, $\pi_M$ is the \evid{factorization homomorphism} of $M$, namely, the unique extension of the identity map on $M$ to a monoid homomorphism $\mathscr F(M) \to M$. Following \cite[Section~3]{Cos-Tri-2023(a)}, we define $\sqeq_M$ as the binary relation on $\mathscr F(M)$ whose graph consists of all pairs $(\mathfrak a, \mathfrak b)$ of $M$-words such that $\mathfrak a$ is, up to associatedness of letters, a subword of some permutation of $\mathfrak b$ (it is readily checked that $\sqeq_M$ is a preorder). Two $M$-words $\mathfrak a$ and $\mathfrak b$ are \evid{equivalent} if $\mathfrak a \sqeq_M \mathfrak b \sqeq_M \mathfrak a$; they are \evid{inequivalent} otherwise. Also, a factorization (into irreducibles) of $x$ is a \evid{minimal factorization} if there is no factorization $\mathfrak b$ of $x$ with $\mathfrak b \sqeq_M \mathfrak a \not\sqeq_M \mathfrak b$.  

We say that $M$ is \evid{factorable} if every non-unit-divisor of $M$ factors as a product of irreducibles; \evid{BF} (or a \evid{bounded factorization monoid}) if $M$ is factorable and the set of all factorization lengths of any fixed element is bounded (the length of a factorization $\mathfrak a$ being the length of $\mathfrak a$ as a word in the free monoid over $M$); \evid{HF} (or a \evid{half-factorial monoid}) if $M$ is factorable and the factorizations of any element have all the same length; \evid{FF} (or a \evid{finite factorization monoid}) if it is factorable and each element has finitely many inequivalent factorizations; and \evid{UF} (or a \evid{unique factorization monoid}) if $M$ is factorable and any two irreducible factorizations of an  element are equivalent.

Replacing factorizations with minimal factorizations in the definition of BF, HF, FF, and UF monoid results in the notions of \evid{BmF}, \evid{HmF}, \evid{FmF}, and \evid{UmF monoid}, resp. Note that, if $M$ is a factorable monoid, then every non-unit-divisor $x \in M$ has at least one minimal factorization, i.e., there exists a factorization $\mathfrak a$ of $x$ with the property that $\mathfrak b \not\sqeq_M \mathfrak a$ for any $\mathscr I(M)$-word $\mathfrak b \in \pi_M^{-1}(x)$.

If we shift our perspective to regard atoms (rather than irreducibles) as the ``building blocks'' of the decompositions of interest, we can modify the above definitions accordingly. This results in the notions of \evid{atomic factorization}, \evid{atomic} monoid, \evid{BF-atomic} monoid, and so on.

\begin{remark}\label{rem:minimal-factors}

In an arbitrary monoid $M$, the only factorization (into irreducibles) of the identity $1_M$ is the empty word. In fact, if $1_M = x_1 \cdots x_n$ for some $x_1, \ldots, x_n \in M$ (with $n \in \mathbb N^+$), then each of the $x_i$'s is a unit-divisor and hence cannot be an irreducible of $M$.

Suppose, on the other hand, that $M$ has the property that every irreducible is a quark. If $a \in \mathscr I(M)$ and $a = x_1 \cdots x_n$ for some non-unit-divisors $x_1, \ldots, x_n \in M$, then $x_i \mid_M a$ for each $i \in \llb 1, n \rrb$. Since $a$ is a quark, it follows that $a$ and each of $x_1, \ldots, x_n$ are associated. This implies $a \sqeq_M x_1 \ast \cdots \ast x_n$, and the inequality is strict unless $n = 1$. So, the only factorization of $a$ (into irreducibles) is the $\mathscr I(M)$-word $a$.

\end{remark}

In the remainder of the section, we will focus our attention on (minimal) factorizations in reduced power monoids. Before proceeding, another remark is in order.

\begin{remark}
\label{rem:minimal-factors-power monoids} 
Let $H$ be an arbitrary monoid.
By Lemma \ref{lem: atoms-irreds}\ref{lem: atoms-irreds(iii)}, two $\mathcal P_{\fin,1}(H)$-words are equivalent if and only if they are a permutation of each other. Thus, a non-empty minimal factorization of a set $X \in \allowbreak \mathcal P_{\fin,1}(H)$ is nothing else than a non-empty word  $A_1 \ast \cdots \ast A_n$ in the free monoid over $\mathscr I(\mathcal P_{\fin,1}(H))$ such that $X \ne \allowbreak A_{\sigma(1)} \cdots \allowbreak A_{\sigma(k)}$ for all $k \in \llb 1, n-1 \rrb$ and every permutation $\sigma$ of $\llb 1, k \rrb$. In addition, it is not difficult to see that a minimal factorization of $X$ must have length $\le |X| - 1$. The proof mirrors the one of \cite[Proposition 4.12]{An-Tr18}, after noticing that, by Remark \ref{rem:minimal-factors} and  Lemma \ref{lem: atoms-irreds}\ref{lem: atoms-irreds(v)}, the only minimal factorization of an irreducible $A \in \mathcal{P}_{\fin,1}(H)$ is the $\mathcal{P}_{\fin,1}(H)$-word $A$ (we omit further details).
\end{remark}

As mentioned in Section~\ref{sec:intro}, minimal factorizations were first introduced by Antoniou and Tringali \cite[Section~4.1]{An-Tr18}, who, however, used atoms (rather than irreducibles) as their building blocks. In particular, this led them to prove in \cite[Theorem 3.9]{An-Tr18} that the atomicity of the reduced power monoid $\mathcal{P}_{\fin,1}(H)$ of a monoid $H$ is equivalent to the condition that $1_H \ne x^2 \ne x$ for all non-identity elements $x \in H$. The next results are complementary to Antoniou and Tringali's theorem and help complete the picture.

\begin{proposition}
\label{prop:2.7}
The following are equivalent for a monoid $H$:

\begin{enumerate}[label=\textup{(\alph{*})}, mode=unboxed]
\item\label{prop:2.7(a)}
$H$ is aperiodic (i.e., every non-identity element of $H$ generates an infinite submonoid).
\item\label{prop:2.7(b)} $ \mathcal{P}_{\fin,1}(H)$ is BF-atomic.
\item\label{prop:2.7(c)} $ \mathcal{P}_{\fin,1}(H)$ is BF.

\item\label{prop:2.7(d)}
$ \mathcal{P}_{\fin,1}(H)$ is FF.
\item\label{prop:2.7(e)} $ \mathcal{P}_{\fin,1}(H)$ is FF-atomic.
\end{enumerate}
\end{proposition}

\begin{proof}
For the equivalence \ref{prop:2.7(a)} $\Leftrightarrow$ \ref{prop:2.7(b)}, see \cite[Theorem 3.11(ii)]{An-Tr18}, where the term ``torsion-free'' is used in place of ``aperiodic'' (but with the same meaning). On the other hand, if $H$ is aperiodic, then $1_H \ne x^2 \ne \allowbreak x$ for every $x\in H\setminus\{1_H\}$, and this ensures by Proposition \ref{prop:2.3} that every irreducible of $\mathcal{P}_{\fin,1}(H)$ is an atom. It is thus clear that \ref{prop:2.7(a)}  $\Leftrightarrow$ \ref{prop:2.7(b)} $\Rightarrow$ \ref{prop:2.7(c)}. Also, it is obvious that any FF-atomic monoid is BF-atomic (cf.~Remark 3.3(2) in \cite{Cos-Tri-2023(a)}), with the result that  \ref{prop:2.7(e)}  $\Rightarrow$ \ref{prop:2.7(b)}. 
So, it remains to show that \ref{prop:2.7(c)} $\Rightarrow$ \ref{prop:2.7(d)} $\Rightarrow$ \ref{prop:2.7(e)}.

\vskip 0.05cm

\ref{prop:2.7(c)} $\Rightarrow$ \ref{prop:2.7(d)}. Let $X \in \mathcal P_{\fin,1}(H)$ and suppose that there exists an integer $k \ge 0$ such that the length of any factorization of $X$ (into irreducibles) is bounded above by $k$. We need to show that the number $n_X$ of factorizations of $X$ in $\mathcal P_{\fin,1}(H)$ is finite. Indeed, since any divisor of $X$ in $\mathcal P_{\fin,1}(H)$ is contained in $X$ and any term in a factorization of $X$ is a divisor of the same set, $n_X$ is bounded above by $|X|^k$ and hence finite (recall from Remark \ref{rem:minimal-factors} that the only factorization of the identity is the empty word).

\vskip 0.05cm

\ref{prop:2.7(d)} $\Rightarrow$ \ref{prop:2.7(e)}. By Lemma \ref{lem: atoms-irreds}\ref{lem: atoms-irreds(ii)}, $\mathcal P_{\fin,1}(H)$ is Dedekind-finite. Since any FF monoid is BF (cf., again, Remark 3.3(2) in \cite{Cos-Tri-2023(a)}), it follows by \cite[Example 5.4(6)]{Cos-Tri-2023(a)} that $\mathcal P_{\fin,1}(H)$ is acyclic. So, every ir\-re\-duc\-i\-ble of $\mathcal P_{\fin,1}(H)$ is an atom (Remark \ref{rem:dedekind-finiteness-and-acyclicity}\ref{rem:dedekind-finiteness-and-acyclicity(1)}) and hence $\mathcal P_{\fin,1}(H)$ is BF-atomic (by the fact that it is BF).
\end{proof}

The limitations inherent in Proposition \ref{prop:2.7} highlight once again that atomic factorizations and, more generally, ``unrestricted factorizations'' into irreducibles are not the best choice possible when it comes to monoids that are not ``nearly cancellative''. The next proposition shows how, at least in the case of reduced power monoids, the paradigm of minimal factorizations allows us to overcome these limitations.

\begin{proposition}
\label{prop:2.8}
The following hold for a monoid $H$:

\begin{enumerate}[label=\textup{(\roman{*})}, mode=unboxed]
\item\label{prop:facts i} $ \mathcal{P}_{\fin,1}(H)$ is FmF.
\item\label{prop:facts ii} BmF-ness and FmF-ness are equivalent properties for $\mathcal{P}_{\fin,1}(H)$.
\end{enumerate}
\end{proposition}

\begin{proof}
This is straightforward from \cite[Theorem 4.11]{Cos-Tri-2023(a)} (see also \cite[Remark 5.5]{Cos-Tri-2023(a)}). 
\end{proof}

Note that, by \cite[Theorem 4.13(a)--(d)]{An-Tr18}, the reduced power monoid of a monoid $H$ is FmF-atomic if and only if it is BmF-atomic, if and only if $1_H \ne x^2 \ne x$ for every non-identity element $x \in H$. This is clearly a consequence of Propositions \ref{prop:2.3} and \ref{prop:2.8}.

\section{Power monoids with unique minimal factorizations}
\label{sect:3}

In the following, we discuss certain conditions for a reduced power monoid to be HmF or UmF. Most notably, we obtain a complete characterization of UmF-ness when the ground monoid is commutative or cancellative. A key role in this direction will be played by the following:

\begin{definition}\label{dfn: almost-brakable}
A (multiplicatively written) semigroup (resp., monoid) $S$ is \evid{almost-breakable} if, for every $x, y \in S$, either $xy \in \allowbreak \{x, y\}$ or $yx \in \{x, y\}$; and it is \evid{breakable} if $xy \in \{x, y\}$ for all $x, y \in S$.
\end{definition}

Breakable semigroups seem to have been first considered by L.~R\'edei in \cite[Section~27]{Red1967}, and they recently played a role in the Bienvenu-Geroldinger isomorphism problem for power monoids \cite[Example 1.2(2)]{Tri-Yan2023(a)}. We have not been able to find any references to almost-breakable semigroups in the literature; consequently, a few elementary remarks may be appropriate before proceeding.

\begin{remarks}\label{rem: almost-breakable}
\begin{enumerate*}[label=\textup{(\arabic{*})},mode=unboxed]
\item\label{rem: almost-breakable 1} Every breakable semigroup is, of course, almost-breakable. The converse need not be true. For instance, let $S$ be the $3$-element magma described by the following (Cayley) table:
$$
\setlength{\extrarowheight}{3pt}
\begin{array}{l|*{3}{l}}
   & x  & y & z \\
\hline
x  & x  & x & x \\
y  & x  & y & x \\
z  & z  & z & z 
\end{array} 
$$
It is easily checked that $S$ is an almost-breakable semigroup. Yet, $S$ is not breakable, since $yz \notin \{y, z\}$.
\end{enumerate*}

\vskip 0.05cm 

\begin{enumerate*}[label=\textup{(\arabic{*})}, resume,mode=unboxed]
\item\label{rem: almost-breakable 2} It is clear from the definition that any almost-breakable monoid $H$ is idempotent, i.e., $x^2=x$ for every $x\in H$. It is then easy to check that $H$ is Dedekind-finite and reduced: In fact, if $xy = 1_H$ for some $x, y \in H$, then $y = xy^2 = xy = 1_H$ and hence $x = y = 1_H$.
    \end{enumerate*}
\end{remarks}

We will also need a special case of the general construction of \evid{ideal extensions} \cite[Section~4.4]{Cli-Pre-1961}. More precisely, let $H$ and $K$ be disjoint semigroups (written multiplicatively). It is then possible to define a ``joint extension'' of the operations of $H$ and $K$ to a binary operation on $H \cup K$ by taking $xy = yx := y$ for all $x\in H$ and $y \in K$. We denote the magma obtained in this way by $H \oslash K$ and call it the \evid{trivial ideal extension of $K$ by $H$}. The next remark highlights some features of this construction.

\begin{remark}
It is easily seen that the trivial ideal extension of a semigroup $K$ by a semigroup $H$ is still a semigroup. In particular, $H \oslash K$ is a monoid if and only if $H$ is, in which case the identity (resp., the group of units) of $H \oslash K$ is the same as the identity (resp., the group of units) of $H$. It follows that $H \oslash K$ is a unit-cancellative monoid if and only if $H$ is a group and $K$ is a cancellative semigroup. By letting $H$ be a non-trivial group and $K$ be a cancellative semigroup, this offers a systematic method for generating unit-cancellative monoids that are not cancellative, which may be interesting in itself when considering that most of the recent work on factorization in a non-cancellative setting has focused on the unit-cancellative case (see, e.g., \cite{Fa-Tr18}, \cite[Sections~4.1 and 4.3]{Tr20(c)}, and \cite{Cos-Tri-2023(b)}). Here, we say that a monoid $M$ is \evid{unit-can\-cel\-la\-tive} if $xy \ne \allowbreak x \ne yx$ for all $x, y \in M$ such that $y$ is a non-unit, whence a commutative monoid is unit-cancellative if and only if it is acyclic.
\end{remark}

With these preliminaries in place, we are about to prove a couple of lemmas that provide stringent \textit{necessary} conditions on the structure of a monoid $H$ whose reduced power monoid is either HmF or UmF. To this end, we recall that the \evid{order} of an element $x \in H$ is the size of the submonoid generated by the element itself (with the result that the order is either a positive integer or $\infty$).

\begin{lemma}\label{lem:3.1}
The following hold for a monoid $H$:
\begin{enumerate}[label=\textup{(\roman{*})}]
\item\label{lem:3.1(i)} If $\mathcal{P}_{\fin,1}(H)$ is HmF, then every element of $H$ has order $\le 3$.
\item\label{lem:3.1(ii)} If $\mathcal{P}_{\fin,1}(H)$ is UmF, then every element of $H$ has order $\le 2$.
\end{enumerate}
In either case, $H$ is periodic (i.e., any element of $H$ has finite order) and hence Dedekind-finite.
\end{lemma}

\begin{proof}
The ``In particular'' part of the statement is straightforward from items \ref{lem:3.1(i)} and \ref{lem:3.1(ii)}, given that every periodic monoid is Dedekind-finite \cite[Remark 4.9(3)]{Tr20(c)}. Therefore, we focus on the rest.

\vskip 0.05cm

\ref{lem:3.1(i)} Let $H$ be non-trivial (otherwise there is nothing to do), suppose $\mathcal{P}_{\fin,1}(H)$ is HmF, and denote by $\kappa$ the order of a non-identity $x \in H$. We have to show that $\kappa \le 3$.
 
Assume to the contrary that $\kappa \ge 4$ and set $X := \{1_H, x, x^2, x^3\}$. It is clear that $x^3 \notin \{1_H, x, x^2\}$, or else $x^n \in \{1_H, x, x^2\}$ for all $n \in \mathbb N$ (by a routine induction) and hence $\kappa \leq 3$ (which is absurd). So, we gather from Proposition \ref{prop:antichains-and-irreds} that the words $\mathfrak a := \{1_H, x\}^{\ast 3}$ and $\mathfrak b := \{1_H, x\} \ast \{1_H, x^2\}$ are both minimal factorizations of $X$ in $\mathcal{P}_{\fin,1}(H)$, with the result that $ \mathcal{P}_{\fin,1}(H)$ is not HmF (a contradiction).

\vskip 0.05cm

\ref{lem:3.1(ii)} Since every UmF monoid is, obviously, HmF (cf.~Remark 3.3(2) in \cite{Cos-Tri-2023(a)}), we gather from \ref{lem:3.1(i)} that any element of $H$ has order $\le 3$, and it only remains to check that this inequality is strict. Suppose to the contrary that $H$ has an element, $x$, of order $3$. Then, similarly as in the proof of the previous item, $x^2 \notin \{1_H, x\}$. Therefore, we get from Proposition \ref{prop:antichains-and-irreds} that the words $\{1_H, \allowbreak x\} \ast \allowbreak \{1_H,x\}$ and $\{1_H,x\} \ast \{1_H,x^2\}$ are inequivalent minimal factorizations of $\{1_H, x, x^2\}$, contradicting that $\mathcal{P}_{\fin,1}(H)$ is a UmF monoid. 
\end{proof}

\begin{lemma}\label{prop: Condizioni necessarie}
Let $H$ be a monoid and suppose that $\mathcal{P}_{\fin,1}(H)$ is UmF. The following hold: 
\begin{enumerate}[label=\textup{(\roman{*})}]
\item\label{prop: Condizioni necessarie(i)} The group of units of $H$ has order $\le 2$.
\item\label{prop: Condizioni necessarie(ii)} $uy=yu=y$ for all $u\in H^\times$ and $y\in H\setminus H^\times$.
\item\label{prop: Condizioni necessarie(iii)} $H\setminus H^\times$ is an almost-breakable semigroup.
\end{enumerate}
\end{lemma}

\begin{proof}
Since $\mathcal{P}_{\fin,1}(H)$ is UmF, we gather from Lemma \ref{lem:3.1}\ref{lem:3.1(ii)} that, for every $x\in H$, either $x^2=1_H$ or $x^2=x$. In particular, $x \in H^\times$ if and only if $x^2=1_H$. Besides that, we will use without further comment that any two-element set in $\mathcal P_{\fin,1}(H)$ is, by Proposition \ref{prop:antichains-and-irreds}, irreducible.

\ref{prop: Condizioni necessarie(i)} Assume by way of contradiction that there exist $u, v \in H^\times$ with $u \ne v$ and $u \ne 1_H \ne v$. Then $uv\notin\{u,v\}$, or else $u=1_H$ or $v=1_H$. Recalling from the above that $u^2 = 1_H$, it follows that
$$
\{1_H,u,v,uv\}=\{1_H,u\}\{1_H,v\}=\{1_H,u\}\{1_H,uv\},
$$
in contradiction to the hypothesis that $\mathcal{P}_{\fin,1}(H)$ is UmF. So, $H^\times$ must have at most two elements.

\ref{prop: Condizioni necessarie(ii)} Let $u\in H^\times$ and $y\in H\setminus H^\times$. We may assume that $H^\times$ is non-trivial, otherwise there is nothing to prove. Then we get from item \ref{prop: Condizioni necessarie(i)} that $H^\times = \{1_H, u\}$ for some $u \in H$ with $u^2 = 1_H \ne u$. If $uy = 1_H$, then $y=uy^2=uy=1_H$ (which is absurd); and if $uy=u$, then $u^2y=y=u^2=1_H$ (which is again absurd). Therefore, $uy$ is a non-unit. Considering that
$$
\{1_H, u\}\{1_H, y\} = \{1_H, u, y, uy\} = \{1_H, u, y, u^2 y\} = \{1_H, u\} \{1_H, uy\}, 
$$
it follows that $\{1_H, u\} \ast \{1_H, y\}$ and $\{1_H, u\} \ast \{1_H, uy\}$ are both minimal factorizations of $\{1_H, u, y, uy\}$. This, however, is only possible if $uy = y$, since $uy \ne y$ would contradict the UmF-ness of $\mathcal P_{\fin,1}(H)$.

\ref{prop: Condizioni necessarie(iii)} Since $H$ is Dedekind-finite by Lemma \ref{lem:3.1}, then $H\setminus H^\times$ is a semigroup. Assume for a contradiction that there exist $x,y\in H\setminus H^\times$ such that $xy, yx \notin \{ x,y\}$. By Lemma \ref{lem:3.1}\ref{lem:3.1(ii)}, $x$ and $y$ are idempotent. So, it is clear that $x \ne y$ and $xy \ne 1_H \ne yx$. Moreover, we gather from Lemma \ref{lem: atoms-irreds}\ref{lem: atoms-irreds(iv)} that the set $X := \{1_H,x,y\}$ is irreducible in $\mathcal{P}_{\fin,1}(H)$. In fact, $\{1_H,x\}$ and $\{1_H,y\}$ are the only proper divisors of $X$, and we have from the above that
$$
\{1_H,x\}\{1_H,y\} = \{1_H, x, y, xy\} \ne X \ne \{1_H, x, y, yx\} = \{1_H,y\}\{1_H,x\}. 
$$
Thus, the words $\{1_H, x\}\ast\{1_H, y\}$ and $\{1_H, x, y\}\ast\{1_H,y\}$ are, by Proposition \ref{prop:antichains-and-irreds}, inequivalent minimal factorizations of $\{1_H,x,y,xy\}$, in contradiction with the UmF-ness of $\mathcal{P}_{\fin,1}(H)$.
\end{proof}

We are finally ready to prove the first main result of the paper, that is, a {\it necessary and sufficient} condition for a reduced power monoid to be UmF.

\begin{theorem}\label{thm:UmFness}
The following are equivalent for a monoid $H$:
\begin{enumerate}[label=\textup{(\alph{*})}]
\item\label{thm:UmFness(i)} $\mathcal{P}_{\fin,1}(H)$ is UmF. 
\item\label{thm:UmFness(ii)} $H \setminus H^\times$ is an almost-breakable subsemigroup of $H$, the group of units of $H$ has order $\le 2$,  $H$ is the trivial ideal extension of $H \setminus H^\times$ by $H^\times$, and $\mathcal{P}_{\fin,1}((H \setminus H^\times) \cup \{1_H\})$ is UmF.
\end{enumerate}
\end{theorem}

\begin{proof}

\ref{thm:UmFness(i)} $\Rightarrow$ \ref{thm:UmFness(ii)}. Since $\mathcal P_{\fin,1}(H)$ is UmF (by hypothesis), we get from Lemma \ref{lem:3.1} that $H$ is Dedekind-finite and hence $K:= (H \setminus H^\times)\cup\{1_H\}$ is a submonoid of $H$. It follows from \cite[Proposition 3.2(iii)]{An-Tr18} that 
$\mathcal{P}_{\fin,1}(K)$ is a divisor-closed submonoid of 
$\mathcal{P}_{\fin,1}(H)$, and then from 
\cite[Proposition 3.5(iii)]{Cos-Tri-2023(a)} that 
$\mathcal{P}_{\fin,1}(K)$ is UmF. By 
Lemma~\ref{prop: Condizioni necessarie}, this suffices to complete the 
first half of the proof.

\ref{thm:UmFness(ii)} $\Rightarrow$ \ref{thm:UmFness(i)}. Let $H \setminus H^\times$ be an almost-breakable subsemigroup of $H$ and set $K := (H \setminus H^\times) \cup \{1_H\}$. If $H^\times$ is trivial, then we have nothing to do. Thus, for the remainder of  
the proof, we suppose that 
$$
H^\times = \{ 1_H, u \},
\qquad\text{for some }
u \in H \text{ with } 
u^2 = 1_H \neq u. 
$$
By definition of a trivial ideal extension, we have that
\begin{equation}
\label{equ:multiplication-by-a-central-element}
\{1_H,u\}Y = Y\{1_H,u\} = Y\cup\{u\},
\qquad\text{for every } Y\in \mathcal{P}_{\fin,1}(H). 
\end{equation}
Consequently, any set of the form $\{1_H, u, y_1, \ldots, y_n\}$ with $y_1, \ldots, y_n \in H \setminus H^\times$ (and $n \in \mathbb N^+$) is not irreducible in $\mathcal P_{\fin,1}(H)$, because it factors as $\{1_H, u\} \allowbreak \{1_H, y_1, \ldots, y_n\}$. In other words, 
\begin{equation}\label{equ:irreds-containing-u}
\text{if } A \text{ is an irreducible of } \mathcal P_{\fin,1}(H) \text{ with } u \in A, \text{ then } A = \{1_H, u\}. 
\end{equation}
Let $X$ be an arbitrary element of $\mathcal{P}_{\fin,1}(H)$. If $u \notin X$, then 
$X$ lies in $\mathcal{P}_{\fin,1}(K)$ and, by the hypothesis that 
$\mathcal{P}_{\fin,1}(K)$ is UmF, it has an essentially unique minimal 
factorization in $\mathcal{P}_{\fin,1}(H)$. So, assume $u \in X$ and let 
$\mathfrak{a} = Y_1 \ast \cdots \ast Y_n$ be a minimal factorization of $X$. Since the product of any two non-units of $H$ is still a non-unit (by the hypothesis that $H \setminus H^\times$ is a subsemigroup of $H$), it is clear that $u \in Y_i$ for some $i \in \allowbreak \llb 1, n \rrb$. It then follows from Eqs.~\eqref{equ:multiplication-by-a-central-element} and \eqref{equ:irreds-containing-u} that, without loss of generality, $Y_1 = \{1_H, u\}$ (recall that two factorizations in $\mathcal{P}_{\fin,1}(H)$ are equivalent if and only if they differ in the order of the factors). 

If $Y_j = \{1_H, u\}$ for some $j \in \llb 2, n \rrb$, then $Y_1 Y_j = \{1_H, u\}^2 = \{1_H, u\} = Y_1$, which, again by  Eqs.~\eqref{equ:multiplication-by-a-central-element} and \eqref{equ:irreds-containing-u}, would contradict the minimality of $\mathfrak a$. Therefore, $Y_j\ne  \{1_H,u\}$ for any $j \in \llb 2, n \rrb$, and $\mathfrak a=\{1_H,u\}\ast \allowbreak Y_2\ast \cdots\ast Y_n$ with $Y_2,\dots,Y_n\in \mathcal{P}_{\fin,1}(K)$.
As a result, $X = (Y_2\cdots Y_n)\cup \{u\}$. Since $u\notin Y_2\cdots Y_n$, it follows that $Y_2\cdots Y_n=X\setminus\{u\}$. Moreover, $\mathfrak a' := Y_2\ast \cdots\ast Y_n$ is a factorization of $X\setminus \{u\}$. 

We claim that $\mathfrak a'$ is actually a minimal factorization of $X \setminus \{u\}$.  
Otherwise, there exists a permutation $\mathfrak b$ of a \textit{proper} subword of $\mathfrak a'$  
such that $\pi(\mathfrak b)=\pi(\mathfrak a')=X \setminus \{u\}$, where $\pi$ denotes the  
factorization hom\-o\-mor\-phism of $\mathcal{P}_{\fin,1}(H)$. But then  
$\pi(\{1_H,u\}\ast\mathfrak b)=X=\pi(\mathfrak a)$, contradicting the minimality of $\mathfrak a$. 

It remains to show that, up to a permutation of its letters, $\mathfrak a$ is the unique  
minimal factorization of $X$. From the above, every minimal factorization of $X$ is of the  
form $\{1_H,u\}\ast\mathfrak z$, where $\mathfrak z$ is a minimal factorization of  
$X\setminus\{u\}$ in $\mathcal{P}_{\fin,1}(K)$. Since $\mathcal{P}_{\fin,1}(K)$ is UmF  
(by hypothesis), it follows that $\mathfrak z$ and $\mathfrak a'$ are equivalent (that is, a permutation of each other), and  
thus $\mathfrak a$ is essentially unique.  
\end{proof}

A straightforward consequence of Theorem \ref{thm:UmFness} is the following characterization,  
which applies in particular to cancellative and acyclic monoids.

\begin{corollary}\label{cor: H cancellative}
Let $H$ be either a group or a monoid whose non-units do not form an almost-breakable semigroup. Then, the reduced power monoid of a monoid $H$ is UmF if and only if $H$ is either trivial or isomorphic to the integers modulo $2$ under addition.
\end{corollary}

Note that, by \cite[Corollary 4.15]{An-Tr18}, a reduced power monoid is UmF-atomic if and only if its ground monoid, say $H$, is trivial. This is a special case of Corollary \ref{cor: H cancellative}. In fact, \cite[Theorem 4.13]{An-Tr18} shows that $\mathcal P_{\fin,1}(H)$ is atomic if and only if $x^2 \notin \{1_H,x\}$ for every non-identity $x \in H$, and the latter condition, in turn, entails that every irreducible of  $\mathcal P_{\fin,1}(H)$ is an atom (Proposition \ref{prop:2.3}), $H \setminus H^\times$ is not an almost-breakable semigroup, and $H$ is not isomorphic to the additive group of integers modulo $2$.

\section{Focus on the almost-breakable case}\label{sec: UmFness and almost-breakable}

By Theorem \ref{thm:UmFness}, characterizing the UmF-ness of $\mathcal{P}_{\fin,1}(H)$ comes down to the case when the ground monoid $H$ is almost-breakable. Through the present section, we thus restrict our attention to this case. Before getting to the punchline, we prove a series of preliminary results. We recall from Section~\ref{subsec:irreds-atoms-quarks} that we write $x \simeq_H y$ to mean that two elements $x$ and $y$ in a monoid $H$ are \textit{associated}, i.e., $HxH = HyH$.

\begin{lemma}\label{lem: almost-breakable} 
Let $H$ be an almost-breakable semigroup and let $x,y\in H$. The following hold:

\begin{enumerate*}[label=\textup{(\roman{*})}]
\item\label{lem: almost-breakable 1} The principal (two-sided) ideals of $H$ form a chain under set inclusion. 
\end{enumerate*}

\vskip 0.05cm 

\begin{enumerate*}[label=\textup{(\roman{*})},resume]
\item\label{lem: almost-breakable 2} If $HxH \subsetneq HyH$, then $xyx=x$ and $xy \simeq_H yx \simeq_H x$.
\end{enumerate*}

\vskip 0.05cm 

\begin{enumerate*}[label=\textup{(\roman{*})},resume]
\item\label{lem: almost-breakable 3} If $x \simeq_H y$, then $\{xy,yx\}=\{x,y\}$. 
\end{enumerate*}

\vskip 0.05cm 

\begin{enumerate*}[label=\textup{(\roman{*})},resume]
\item\label{lem: almost-breakable 4} There are no elements $x,y,z\in H$ with $HxH \subsetneq HyH$ and $HxH \subsetneq HzH$ such that $xy \ne x \ne zx$.
\end{enumerate*}

\vskip 0.05cm 

\begin{enumerate*}[label=\textup{(\roman{*})},resume]
\item\label{lem: almost-breakable 5} If $HxH \subsetneq HyH$, $yx \ne x$, and $x \simeq_H x'$ for some $x'\in H$, then $x'y=x'$.
\end{enumerate*}
\end{lemma}

\begin{proof}
Recall from Remark \ref{rem: almost-breakable}\ref{rem: almost-breakable 2} that every almost-breakable semigroup is idempotent, and note that the opposite monoid of an almost-breakable monoid is itself almost-breakable. We will freely use these simple facts in the remainder without further comment.

\vskip 0.05cm

\ref{lem: almost-breakable 1} Since $H$ is almost-breakable, either $xy\in \{x,y\}$ or $yx\in\{x,y\}$. Consequently, either $x \in \{xy, yx\}$ or $y \in \{xy, yx\}$. In the former case, $HxH\subseteq HyH$; and in the latter, $HyH \subseteq HxH$.

\vskip 0.05cm

\ref{lem: almost-breakable 2} By the almost-breakability of $H$, either $xy\in\{x,y\}$ or $yx \in \allowbreak \{x,y\}$. Assume  $xy \in \{x,y\}$ (the other case is analogous). Then, $xy=x$ (otherwise $y\in HxH$), and so $xyx=x$. It follows that $x\in HxyH$ and $x \in \allowbreak HyxH$, so $HxH \subseteq HxyH$ and $HxH \subseteq HyxH$. The reverse inclusions are trivial.

\vskip 0.05cm
    
\ref{lem: almost-breakable 3} Assume without loss of generality that $xy=x$ (otherwise, flip the roles of $x$ and $y$, or work in the opposite monoid of $H$). Then $xyx=x^2 = x$ and, being $HxH=HyH$, $y=uxv$ for some $u, v\in H$. It follows that $yv=yxv=y$ and so $yxy=yxyv=yxv=y$. Thus, $y=yxy=yx$.

\vskip 0.05cm
    
\ref{lem: almost-breakable 4} Assume the contrary. We know from \ref{lem: almost-breakable 1} that $HyH \subseteq HzH$ or $HzH \subseteq HyH$. So, we get from \ref{lem: almost-breakable 2} and \ref{lem: almost-breakable 3} that $yz\in HyH \cup HzH \supsetneq HxH$. Moreover, the hypotheses imply by \ref{lem: almost-breakable 2} that $xy \simeq_H x \simeq_H zx$. Thus, we conclude from \ref{lem: almost-breakable 3} that $\{xy, zx\}\ni (xy)(zx)=x(yz)x=x$, contradicting that $xy\ne x\ne zx$.

\vskip 0.05cm
    
\ref{lem: almost-breakable 5} Assume to the contrary that there exists $x'\in H$ such that $HxH=Hx'H$ and $x'y\ne x'$. It is then clear that $x'y\ne y$, or else $HxH \subsetneq HyH \subseteq Hx'H=HxH$ (which is absurd). So, by the almost-breakability of $H$, we have by \ref{lem: almost-breakable 2} that $yx'=x'$. Analogously, we get from the hypothesis that $xy=x$. Thus $x'x=yx'x=x'xy$. Since $x\simeq_H x'$, by \ref{lem: almost-breakable 3}, either $x'x=x$ or $x'x=x'$. If $x'x=x$, then $x=x'x=y(x'x)=yx$, while if $x'x=x'$, then $x'=x'x=(x'x)y=x'y$. This is a contradiction.
\end{proof}

Next, we characterize the irreducibles of the reduced power monoid of an almost-breakable monoid and describe an interesting property of its factorizations.

\begin{proposition}\label{prop: irreds and square-free}
The following hold for an almost-breakable monoid $H$:

\begin{enumerate}[label=\textup{(\roman{*})}, mode=unboxed]

\vskip 0.05cm

\item\label{prop: irreds and square-free i} The irreducibles of $\mathcal P_{\fin,1}(H)$ are precisely the $2$-element subsets of $H$ containing the identity.

\vskip 0.05cm

\item\label{prop: irreds and square-free ii} Every element of $\mathcal P_{\fin,1}(H)$ admits a square-free factorization into irreducibles.
\end{enumerate}
\end{proposition}
\begin{proof}
\ref{prop: irreds and square-free i} We may restrict our attention on the ``only'' part of the statement, because we already know from Prop\-o\-si\-tion \ref{prop:antichains-and-irreds} that any $2$-element set in $\mathcal P_{\fin,1}(H)$ is irreducible. 

Fix $X \in \mathcal P_{\fin,1}(H)$ with $|X| \ge 3$. We need to show that $X$ is not irreducible. To begin, we gather from Lemma \ref{lem: almost-breakable}\ref{lem: almost-breakable 1} (and the fact that $X$ is finite and non-empty) that there is an element $x\in X$ such that $HxH \subseteq HyH$ for every $y \in X$. Note that $x\ne 1_H$, as we are guaranteed by Remark \ref{rem: almost-breakable}\ref{rem: almost-breakable 2} that $HyH = \allowbreak H$ if and only if $y=1_H$. 
Accordingly, we define the following subsets of $X$:
\begin{itemize}
\item $Y := \{y \in X \colon y \simeq_H x\} \cup \{1_H\}$;
\item $A := \{u \in X \colon u \not\simeq_H x \text{ and } uy = y \text{ for all } 1_H \ne y \in Y\}$;
\item $B := \{v \in X \colon v \not\simeq_H x \text{ and } vy \ne y \text{ for some }1_H\ne y\in Y\}\cup\{1_H\}$.
\end{itemize}
It is clear that $1_H\in A$ and $A,B\subseteq (X\setminus Y)\cup \{1_H\}\subsetneq X$. Moreover, 
\begin{equation}\label{eq: unions}
   X = Y \cup A \cup B, \qquad Y \cap A = Y \cap B = A \cap B = \{1_H\},
   \qquad\text{and}\qquad
   AY = A \cup Y.
\end{equation}
We will distinguish two cases, depending on whether $B = \{1_H\}$ or $B \ne \{1_H\}$.

\vskip 0.05cm

\textsc{Case 1:} $B=\{1_H\}$. If $A=\{1_H\}$, then we have from Eq.~\eqref{eq: unions} that
\begin{equation}\label{B=A=1}
X=Y=\{1_H,y_1,\dots,y_k\}
\end{equation}
with $k \ge 2$ and $y_i \simeq_H x$ for every $i \in \llb 1, k \rrb$ (note that $x = y_j$ for some $j \in \llb 1, k \rrb$).
We then get from Lemma \ref{lem: almost-breakable}\ref{lem: almost-breakable 3} that $y_i y_j \in \{y_i, y_j\}$ for every $i, j \in \llb 1, k \rrb$. Thus $X = \{1_H, y_1\}\{1_H, y_2, \ldots y_k\}$ and, by definition, it is not irreducible.
If, on the other hand, $A\ne \{1_H\}$, then Eq.~\eqref{eq: unions} yields
\begin{equation}\label{B=1 ne A}
X = AY
\end{equation}
and there exists $u \in X\setminus \{1_H\}$ with $u\not\simeq_H x$ such that $uy = y$ for all $y \in Y \setminus \{1_H\}$. This implies that $A \subsetneq X$ ($x \notin A$) and $Y \subsetneq X$ ($u \notin Y$), so again $X$ is not irreducible.

\vskip 0.05cm

\textsc{Case 2:} $B\ne\{1_H\}$. There exist $b \in B \setminus\{1_H\}$ and $y \in Y \setminus \{1_H\}$ such that $by \ne y$. We claim that
\begin{equation}\label{B ne 1}
X=\{1_H,y\}(A\cup B\cup Y'),
\end{equation}
where $Y':=Y\setminus\{y\}$. Let $z \in X\setminus\{1_H\}$. Then either $z \simeq_H x$ (and in this case $z \in Y$) or $z \not\simeq_H x$ (and in this case either $z \in A$ or $z \in B$). It then follows that $X \subseteq \{1_H, y\}(A \cup B\cup Y \setminus \{y\})$. As for the reverse inclusion, it is enough to prove that each of $yA$, $yB$, and $yY'$ is a subset of $X$. Let $u\in A$ and assume for a contradiction that $yu \notin X$. Then $yu \ne y \ne by$, which contradicts Lemma \ref{lem: almost-breakable}\ref{lem: almost-breakable 4}. Thus, $yA \subseteq X$. Now let $v \in B$. We claim that $yv = y$. If $v = 1_H$, there is nothing to prove. Otherwise, we have from the definition itself of $B$ that $v \not\simeq_H x$ and $vy' \ne y'$ for some $y' \in Y \setminus\{1_H\}$. It then follows by Lemma \ref{lem: almost-breakable}\ref{lem: almost-breakable 5} that $yv = y$ and hence $yB \subseteq X$. Lastly, by Lemma \ref{lem: almost-breakable}\ref{lem: almost-breakable 3}, $yz\in \{y, z\}$ for every $z \in Y$, so that $yY'\subseteq Y$. Note that the set $\{1_H\}\ne A \cup B\cup Y'$ in \eqref{B ne 1} is a proper subset of $X$ (because $y \notin A\cup B \cup Y'$), so $X$ is not irreducible.

\vskip 0.05cm

\ref{prop: irreds and square-free ii} Pick a set $X \in \mathcal P_{\fin,1}(H)$ and set $n := |X|-1$. If $n=0$ or $n=1$ the conclusion is trivial. So, suppose $n\ge 2$ and assume inductively that every $Z\in \mathcal P_{\fin,1}(H)$ with $|Z|\le n$ admits a square-free factorization (into irreducibles). Let $x$ be a $\mid_H$-maximal element in $X$ and decompose $X$ as done in \ref{prop: irreds and square-free i}. If $B=A=\{1_H\}$, then \eqref{B=A=1} holds and $X=Y=\{1_H,y_1,\dots,y_k\}$ with $k\ge 2$ and $y_i\simeq_H x$ for every $i$. In light of Lemma \ref{lem: almost-breakable}\ref{lem: almost-breakable 3}, we then get that $X=\{1_H,y_1\}\cdots\{1_H, y_k\}$ is a square-free factorization of $X$. If $B=\{1_H\}\ne A$, then \eqref{B=1 ne A} holds and $X=AY$ with $\{1_H\}\ne A,Y\subsetneq X$. By the inductive hypothesis, each of $A$ and $Y$ admits a square-free factorization and, since $A \cap Y=\{1_H\}$ and the factors of a set $Z\in \mathcal P_{\fin,1}(H)$ are all subsets of $Z$, we can conclude that $X$ itself admits a square-free factorization. Lastly, if $B\ne\{1_H\}$, then \eqref{B ne 1} holds and $X=\{1_H,y\}(A\cup B\cup Y')$, where $Y':=Y\setminus\{y\}$ and $y\in X$ is such that $y\simeq_H x$. Since $y\notin A\cup B\cup Y'$, we get again by the inductive hypothesis that the set $A\cup B\cup Y'$ (and hence $X$) admits a square-free factorization.
\end{proof}

\begin{corollary}\label{cor: square-free}
    If the reduced power monoid of an almost-breakable monoid is UmF, then the minimal factorizations of its elements are square-free.
\end{corollary}

\begin{proof}
    Let $H$ be an almost-breakable monoid. If $\mathcal P_{\fin,1}(H)$ is UmF, then the minimal factorizations of a set  $X \in \mathcal P_{\fin,1}(H)$ have all the same factors (by Remark \ref{rem:minimal-factors-power monoids}). So, it remains to see that $X$ has a square-free minimal fac\-tor\-i\-za\-tion. For, we know from Proposition \ref{prop: irreds and square-free}\ref{prop: irreds and square-free ii} that $X$ admits a square-free factorization $\mathfrak a$. If $\mathfrak a$ is minimal, then we are done. Otherwise, there is a subword $\mathfrak b$ of $\mathfrak a$ which is a minimal factorization of $X$, and it is obvious that a subword of a square-free word is itself square-free.
\end{proof}

We introduce the following definitions to streamline the sequel of the presentation.
\begin{definition}\label{dfn: twisted and bridged}
We let a \evid{balanced pair} of a monoid $H$ be an ordered pair $(x,y) \in H \times H$ such that $xy\in \{x, y\}$; otherwise, the pair is \evid{unbalanced}. Accordingly, the monoid $H$ is \evid{twisted} if there exist unbalanced pairs $(x,y)$ and $(z,w)$ of $H$ with $\{x,y\}\cap\{z,w\}=\emptyset$ such that $xy \in \{z,w\}$ and $zw\in\{x,y\}$; and \evid{bridged} if there exist unbalanced pairs $(x_1,x_2)$, $(x_2,x_3)$, and $(x_1,x_3)$ of $H$ such that $x_1x_3 \notin \{x_1 x_2, x_2 x_3\}$.
If $H$ is not twisted (resp., bridged), then we call it \evid{untwisted} (resp., \evid{unbridged}).
\end{definition}

Our interest in these concepts arises from the next proposition, in which we show that the ground monoid being almost-breakable is not \textit{sufficient} for a reduced power monoid to be UmF. More precisely, we are going to see that, if $H$ is an almost-breakable monoid, then a \textit{necessary} condition for $\mathcal{P}_{\fin,1}(H)$ to be UmF is that $H$ is also untwisted {\it and} unbridged. To this end, we will often use without further mention that, if $H$ is almost-breakable and $(x,y)$ is an unbalanced pair of $H$, then the pair $(y,x)$ is balanced.

\begin{proposition}\label{prop: P}
    Let $H$ be an almost-breakable monoid. If $H$ is either twisted or bridged, then the reduced power monoid of $H$ is not UmF.
\end{proposition}
\begin{proof}
Assume first that $H$ is twisted. Accordingly, let $(x,y)$ and $(z,w)$ be two unbalanced pairs of $H$ with $\{x,y\}\cap\{z,w\}=\emptyset$ such that $xy \in \{z,w\}$ and $zw\in\{x,y\}$. We claim that the set $X:=\{1_H,x,y,z,w\}$ has two inequivalent minimal factorizations (and hence $\mathcal{P}_{\fin,1}(H)$ is not UmF). To this end, we distinguish two cases that cover all possible configurations. 
    
    {\sc Case 1:} $xy=z$ and $zw=x$.
     Under these assumptions, $(z,x)$, $(x,z)$, $(z,y)$, $(y,z)$, $(x,w)$, $(w,x)$, $(y,x)$, $(w,z)$ are balanced pairs. Moreover, by the hypothesis of almost-breakability, either $(y,w)$ or $(w,y)$ is balanced. Say $yw\in \{y,w\}$ (the other case is similar). Then,
    $\{1,x\}\ast\{1,y\}\ast\{1,w\}$ and $\{1,y\}\ast\{1,z\}\ast\{1,w\}$ are inequivalent minimal factorizations of $X$. The minimality follows from a purely combinatorial argument, since the product of two $2$-element sets contains at most $4$ elements while $|X|=5$.
    
    {\sc Case 2:} $xy=z$ and $zw=y$.
    Under these assumptions, $(z,x)$, $(x,z)$, $(z,y)$, $(y,z)$, $(y,w)$, $(w,y)$, $(y,x)$, $(w,z)$ are balanced pairs. Moreover, by almost-breakability, either $(x,w)$ or $(w,x)$ is also balanced. Say $wx\in \{x,w\}$ (the other case is analogous). Then, $\{1,x\}\ast\{1,z\}\ast\{1,w\}$ and $\{1,w\}\ast\{1,x\}\ast\{1,y\}$ are inequivalent minimal factorizations of $X$.

  \vskip 0.05cm

  Assume now that $H$ is bridged. Let $x_1,x_2,x_3\in H$ be such such that $(x_1,x_2)$, $(x_2,x_3)$, and $(x_1,x_3)$ are unbalanced pairs and $x_1 x_2 \ne x_1x_3 \ne x_2 x_3$. We claim that 
  \[\mathfrak a:=\{1_H, x_2\}\ast\{1_H,x_3\}\ast\{1_H,x_1\}\ast\{1_H,x_2\}\] 
  is a minimal factorization of the set 
  \[Y:=\{1_H, x_1,\allowbreak x_2, \allowbreak x_3,\allowbreak x_1 x_2, x_2x_3\}.\] 
  By Corollary \ref{cor: square-free}, this would imply that $\mathcal{P}_{\fin,1}(H)$ is not UmF. In light of items \ref{lem: almost-breakable 1} and \ref{lem: almost-breakable 3} of Lemma \ref{lem: almost-breakable}, the two-sided ideals $Hx_iH$ are ordered with respect to the strict inclusion $\subsetneq$. Thus, by Lemma \ref{lem: almost-breakable}\ref{lem: almost-breakable 2}, $x_{2}x_3\ne x_1$ and $x_{1} x_{2}\ne x_{3}$. Moreover, $x_1 x_2\ne x_2 x_3$. In fact, by Lemma \ref{lem: almost-breakable}\ref{lem: almost-breakable 4}, if $x_1x_2=x_2x_3$, then it must be either $H x_1 H\subsetneq H x_2 H$ or $H x_3 H\subsetneq H x_2 H$. If, e.g., $H x_1 H\subsetneq H x_2 H$ (the other case works analogously), then $x_1x_2\simeq_H x_1$ but $x_2x_3=x_1x_2$ is associated to $x_2$ or $x_3$, a contradiction. It then follows that $|Y|=6$. We first observe that 
  \[Y=\{1_H, x_2\}\{1_H,x_3\}\{1_H,x_1\}\{1_H,x_2\},\] 
  i.e., $\mathfrak a$ is a factorization of $Y$. In fact, since $H$ is almost-breakable, then $(x_2, x_1), (x_3,x_2),(x_3,x_1)$ are balanced pairs and, as a consequence, $x_2x_1x_2\in \{x_1x_2,x_2\}$, $x_2x_3x_2\in \{x_2x_3,x_2\}$, $x_2x_3x_1\in\{x_1,x_2,x_2x_3\}$, $x_3x_1x_2\in\{x_2,x_3, x_1x_2\}$, and $x_2x_3x_1x_2\in\{x_2,x_1x_2,x_2x_3\}$. It remains to prove that $\mathfrak a$ is minimal. Note that $\{1_H, x_1\},\{1_H,x_2\},\{1_H,x_3\}$ must appear in all possible factorizations of $Y$ since it is impossible to recover one of the $x_i$'s as a product of the other two. Moreover, the factor $\{1_H,x_3\}$ must always precede $\{1_H,x_1\}$, otherwise the element $x_1x_3\notin Y$ would appear in the product. Therefore, the only subwords of $\mathfrak a$ that can be eligible as minimal factorizations of $Y$ are 
  \[\mathfrak b:=\{1_H, x_2\}\ast\{1_H,x_3\}\ast\{1_H,x_1\}, \quad \mathfrak c:= \allowbreak \{1_H, \allowbreak x_3\}\ast\{1_H, x_2\}\ast\{1_H,x_1\},\]
  \[\text{and } \mathfrak d:=\{1_H,x_3\}\ast\{1_H,x_1\}\ast\{1_H, x_2\}.\] 
  However, it is immediate to verify by a direct computation that $\pi(\mathfrak b)$ and $\pi(\mathfrak c)$ do not contain $x_1x_2$, while $\pi(\mathfrak d)$ does not contain $x_2x_3$. This shows that $\mathfrak a$ is minimal and concludes the proof.
\end{proof}

The examples below show that, for an almost-breakable monoid, twisted-ness and bridged-ness are independent properties, which is not immediately clear from the definitions.

\begin{examples}
\begin{enumerate*}[label=\textup{(\arabic{*})}, mode=unboxed]
\item Let $H_1$ be the unitization of the $4$-element magma defined by the following table:
$$
\setlength{\extrarowheight}{3pt}
\begin{array}{l|*{4}{l}}
   & x_1  & x_2 & x_3 & x_4\\
\hline
x_1  & x_1  & x_1 & x_1 & x_1 \\
x_2 & x_1  & x_2 & x_1 & x_2 \\
x_3  & x_3  & x_3 & x_3 & x_3 \\
x_4 & x_3 & x_4 & x_3 & x_4
\end{array}
\;\;.
$$
It is readily verified that $H_1$ is an almost-breakable monoid. In addition, $H_1$ is twisted and unbridged, because $(x_2,x_3)$ and $(x_4,x_1)$ are the only unbalanced pairs of $H_1$ and we have $x_2x_3=x_1$ and $x_4x_1=x_3$.
\end{enumerate*}

\vskip 0.05cm

\begin{enumerate*}[label=\textup{(\arabic{*})}, mode=unboxed, resume]
\item Let $H_2$ be the unitization of the $6$-element magma described by the following table:
$$
\setlength{\extrarowheight}{3pt}
\begin{array}{l|*{6}{l}}
   & x_1  & x_2 & x_3 & x_4 & x_5 & x_6\\
\hline
x_1  & x_1  & x_1 & x_1 & x_1 & x_1 & x_1\\
x_2 & x_1  & x_2 & x_1 & x_1 & x_2 & x_6 \\
x_3  & x_3  & x_3 & x_3 & x_3 & x_3 & x_3\\
x_4 & x_4 & x_4 & x_4 & x_4 & x_4 & x_4\\
x_5 & x_1 & x_2 & x_3 & x_3 & x_5 & x_6\\
x_6 & x_1 & x_2 & x_1 & x_1 & x_2 & x_6
\end{array}
\;\;.
$$
It is tedious but easy to check that $H_2$ is an untwisted almost-breakable monoid. Moreover, $H_2$ is bridged, since $(x_6, x_5), (x_5, x_4), (x_6,x_3)$ are unbalanced pairs with $x_6x_3=x_1\notin \{x_6x_5, x_5x_4\}=\{x_2,x_3\}$.
\end{enumerate*}
\end{examples}

The next proposition shows that, while the monoid $H$ being almost-breakable is not a sufficient condition for $\mathcal{P}_{\fin,1}(H)$ to be UmF, $H$ being breakable is.

\begin{proposition}\label{prop: breakable}
The reduced power monoid of a breakable monoid is UmF.
\end{proposition}

\begin{proof}
Let $H$ be a breakable monoid and fix $X \in \mathcal P_{\fin,1}(H)$. We get from Proposition \ref{prop:2.8}\ref{prop:facts i} that $\mathcal{P}_{\fin,1}(H)$ is FmF (and hence factorable), so it only remains to check that $X$ has an essentially unique minimal fac\-tor\-i\-za\-tion (into irreducibles). For, assume $k := |X| - 1 \ne 0$ and write $X = \{1_H, x_1, \ldots, x_k\}$, where the $x_i$'s are non-identity elements of $H$ (if $k = 0$, then $X = \{1_H\}$ and we are done, see Remark \ref{rem:minimal-factors}). 

If $A$ is an irreducible divisor of $X$, then it follows from Proposition \ref{prop: irreds and square-free}\ref{prop: irreds and square-free ii} and Lemma \ref{lem: atoms-irreds}\ref{lem: atoms-irreds(ii)} that $A = \allowbreak \{1_H, x_{i}\}$ for some $i\in \llb 1, k \rrb$. So, we gather from Remark \ref{rem:minimal-factors-power monoids} that a minimal factorization of $X$ must be a $\mathcal P_{\fin,1}(H)$-word of the form $\mathfrak a = \{1_H, x_{i_1}\} \ast \cdots \ast \{1_H, x_{i_n}\}$ with $n \in \llb 1, k \rrb$ and $i_1, \ldots, i_n \in \llb 1, k \rrb$. To conclude, it is therefore enough to check that $n$ equals $k$ and the $i_j$'s are pairwise distinct. 

To begin, denote by $\pi$ the factorization homomorphism of $\mathcal P_{\fin,1}(H)$ and suppose to the contrary that $n < k$. There then exists an index $i \in \llb 1, k \rrb$ such that $i \notin I := \{i_1, \ldots, i_n\}$. Since $x_i \in X \setminus \{1_H\}$ and $X = \pi(\mathfrak a)$, we thus find that $x_i = x_{j_1} \cdots x_{j_r}$ for some $r \in \llb 1, n \rrb$ and $j_1, \ldots, j_r \in I$. But this is impossible, because $H$ being breakable implies $x_{j_1} \cdots x_{j_r} \in \{x_{j_1}, \ldots, x_{j_r}\} \subseteq X \setminus \{x_i\}$. 

Next, assume for a contradiction that there exist $s, t \in \llb 1,n\rrb$ with $s < t$ such that $x_{i_s}=x_{i_t}$. Accordingly,  
set $A := \pi(\mathfrak z_1)$ and $B := \pi(\mathfrak z_2)$, where the $\mathcal P_{\fin,1}(H)$-words $\mathfrak z_1$ and $\mathfrak z_2$ are defined by
$$
\mathfrak{z}_1 := 
\left\{
\begin{array}{ll}
\{1_H, x_{i_1}\} \ast \cdots \ast \{1_H, x_{i_{s-1}}\} & \text{if }s \ne 1, \\
1_{\mathscr F(H)} & \text{if } s = 1,
\end{array}
\right.
\qquad\text{and}\qquad
\mathfrak{z}_2 :=
\{1_H, x_{i_{s+1}}\}\ast \allowbreak \cdots \ast  \{1_H, x_{i_{n}}\}.
$$
Then $\pi(\mathfrak a)=A\{1,x_s\}B=AB\cup A x_s B=AB=\pi(\mathfrak{z}_1\ast\mathfrak{z}_2)$ and this contradicts the minimality of $\mathfrak a$.
\end{proof}

We finish with a characterization of UmF-ness in the reduced power monoid of a commutative monoid.

\begin{theorem}\label{thm: H commutative}
   The following are equivalent for a commutative monoid $H$:
    \begin{enumerate}[label=\textup{(\alph{*})}]
    \item\label{thm:H commutative(i)} $\mathcal{P}_{\fin,1}(H)$ is UmF. 
    \item\label{thm:H commutative(ii)} $H \setminus H^\times$ is a breakable subsemigroup of $H$, the group of units of $H$ has order $\le 2$, and $H$ is the trivial ideal extension of $H \setminus H^\times$ by $H^\times$.
    \end{enumerate}
\end{theorem}
\begin{proof}
    A commutative monoid $H$ is almost-breakable if and only it is breakable. The claim is therefore an immediate consequence of Theorem \ref{thm:UmFness} and Proposition \ref{prop: breakable}.
\end{proof}

\section*{Acknowledgments}
L.~Cossu is a member of the National  
Group for Algebraic and Geometric Structures and their Applications (GNSAGA) of the Italian  
Mathematics Research Institute (INdAM). She acknowledges support from the European Union's Horizon 2020 program through the  
Marie Sk\l{}odowska-Curie grant agreement no.~101021791, from the Austrian Science Fund  
(FWF) through project PAT-9756623, and from the project ``FIATLUCS'' funded by the PNRR  
RAISE Liguria, Spoke 01 (CUP: F23C24000240006). 

The Marie Sk\l{}odowska-Curie grant also supported  
S.~Tringali's visit to the University of Graz in the summer of 2023, when this paper was  
initiated. In addition, S.~Tringali acknowledges support from grant A2023205045, funded by  
the Natural Science Foundation of Hebei Province.  

The authors are both grateful to Pace Nielsen (Brigham Young University, US) for suggesting the term ``trivial ideal extension'' (see the comments to \url{https://mathoverflow.net/questions/448918/}), to Benjamin Steinberg (City University of New York, US) for valuable insights on almost-breakable semi\-groups (see \url{https://mathoverflow.net/questions/450225/}), and to the anonymous referee of an earlier version of this paper for their careful reading and helpful suggestions.

\end{document}